\documentclass[reqno,a4paper]{amsart}

\usepackage{mathrsfs,amsmath,amssymb,graphicx,fancyhdr}
\usepackage[percent]{overpic}
\usepackage[foot]{amsaddr}
\usepackage{stmaryrd}
\usepackage[unicode]{hyperref}
\usepackage{latexsym}
\usepackage{layout}
\usepackage[english]{babel}
\usepackage{color}
\usepackage{subfigure}
\usepackage{fancyvrb}
\usepackage[all]{xypic}

\numberwithin{equation}{section}
\theoremstyle{plain}
\newtheorem{theorem}{Theorem}[section]
\newtheorem{Example}[theorem]{Example}
\newtheorem{lemma}[theorem]{Lemma}
\newtheorem{Proposition}[theorem]{Proposition}
\newtheorem{corollary}[theorem]{Corollary}

\newtheorem{definition}[theorem]{Definition}
\theoremstyle{remark}
\newtheorem{remark}[theorem]{Remark}

\newcommand{\projHM}{\projH_{H,\mfd}}

\newcommand{\modtwo}{\;(\mathrm{mod }~2)}

\newcommand{\projH}{\mathfrak{p}}
\newcommand{\piunoa}{\pi_1(\mfda,\polo)}
\newcommand{\piuno}{\pi_1(\mfd,\polo)}
\newcommand{\lkn}{\mathrm{link}}

\newcommand{\dabcv}{d_{\abcv}}
\newcommand{\abcv}{\mfd_H}
\newcommand{\abcva}{\mfda_H}
\newcommand{\mfda}{\mfd}
\newcommand{\lacci}{C_{x_0}([0,1]; \mfd)}
\newcommand{\ucov}{\widetilde \mfda}
\newcommand{\uproj}{\projH}
\newcommand{\la}     {\lambda}

\newcommand{\polo}    {x_0}
\newcommand{\ga}    {\gamma}
\newcommand{\arr} {\arrow}
\newcommand{\base} {M}

\newcommand{\chart} {D}
\newcommand{\cut} {\Sigma}
\newcommand{\Part}{\mathrm{Part}}
\newcommand{\Sf}         {\ensuremath{\mathbb S}}
\newcommand{\isom}{f_{\mathbf{\cut}}}
\newcommand{\cvv}{Y_ \mathbf{\cut}}
\newcommand{\BVc}{\domain(\FFF)}
\newcommand{\cutpair}{\mathbf{\cut}}
\newcommand{\newcutpair}{\mathbf{\Gamma}}

\newcommand{\valoreminimocut}{\mathscr A(\cutpair)}

\newcommand{\eps}       {\epsilon}

\newcommand{\hn}{\mathcal H^{\dime-1}}
\newcommand{\dime}{n}

\newcommand{\cover} {Y}

\newcommand{\dde}           {\partial}
\newcommand{\uw}{v}

\newcommand{\mfd} {M}

\newcommand{\pitilde}{ \widetilde{\pi} }
\newcommand{\CutsmenoS}{\cut\setminus S}

\newcommand{\uv}{v}
\newcommand{\id} {()}

\newcommand{\dcv}{d_{\cvv}}

\newcommand{\Hole} {S}
\newcommand{\pione} {\pi_1(\base)}

\newcommand{\unionedisgiunta} {\mathcal X}

\newcommand{\edges} {S}
\newcommand{\looops} {C}
\newcommand{\iwire} {\looops}

\newcommand{\Om} {\Omega}

\newcommand{\domain} {D}

\newcommand{\domainF} {\domain(\FFF)}

\newcommand{\FFF} {\mathcal F}
\newcommand{\R} {\mathbb R}

\newcommand{\SSS} {\mathcal S}
\newcommand{\prj} {p}

\newcommand{\Z} {\mathbb Z}

\definecolor{mygray}{rgb}{0.92,0.92,0.92}

\newif\ifdraft
\draftfalse

\title{Covers, soap films and BV functions}

\author{Giovanni Bellettini}
\address{Dipartimento di Ingegneria dell'Informazione e Scienze Matematiche, Universit\`a di Siena, 53100 Siena, Italy,
and International Centre for Theoretical Physics ICTP, 
Mathematics Section, 34151 Trieste, Italy
}
\email{bellettini@diism.unisi.it}
\author{Maurizio Paolini}
\address{Dipartimento di Matematica e Fisica, Universit\`a Cattolica del Sacro Cuore, 25121 Brescia, Italy}
\email{maurizio.paolini@unicatt.it}
\author{Franco Pasquarelli}
\address{Dipartimento di Matematica e Fisica, Universit\`a Cattolica del Sacro Cuore, 25121 Brescia, Italy}
\email{franco.pasquarelli@unicatt.it}
\author{Giuseppe Scianna}
\address{Dipartimento di Ingegneria dell'Informazione e Scienze Matematiche, Universit\`a di Siena, 53100 Siena, Italy}
\email{giuseppe.scianna@unisi.it}

\date{\today}

\begin{document}

\thanks{}

\begin{abstract}
In this paper we review 
the double covers  method with constrained BV
functions for solving the classical  Plateau's problem.
Next, we carefully analyze some interesting examples
of soap films compatible with covers
of degree larger than two: in particular,
the case of a soap film only partially wetting a space curve, 
a soap film spanning a cubical frame but having a large tunnel, 
a soap film that retracts onto its
boundary, hence not modelable with the Reifenberg method,
and various soap films spanning an octahedral frame.
\end{abstract}

\maketitle


\section{Introduction}\label{sec:intro}
In \cite{Br:95} K. Brakke  introduced the covering space method
for solving a rather large class of one-codimensional Plateau type problems,
including the classical case of an area-minimizing surface
spanning a knot, a Steiner minimal graph connecting a given
number of points in the plane, and an area-minimizing surface
spanning a nonsmooth one-dimensional frame such as the 
one-skeleton of a polyhedron. The method 
does not impose any  topological
restriction on the solutions; it
relies on the theory of currents and takes into
account also unoriented objects.
It consists essentially in the construction of a pair of covering spaces, and is based on 
the minimization of what the author called the soap film mass.

Recenlty, a slightly different approach
has been proposed in \cite{AmBePa:17}; it is
based on the minimization of the total
variation for functions defined on a single covering space and
satisfying a suitable contraint on the fibers. 
Also this method 
does not impose any  topological
restriction on the solutions. 
Moreover, it
takes advantage of the full
machinery 
known on the space of BV functions defined on a locally Euclidean manifold: for instance,
and remarkably,
it allows approximating the considered class of Plateau type problems by 
$\Gamma$-convergence. In the forthcoming paper \cite{BePaPaSc:17} 
we shall deepen this 
$\Gamma$-convergence regularization 
for finding minimal networks in the plane.

The interest in the covering space method is also illustrated in 
the recent paper \cite{BePaPa:17}, where is shown  
a triple cover of $\R^3 \setminus (S\cup C)$, $S$ a tetrahedral frame and $C$ two disk boundaries,
compatible with a soap film spanning $S$
and having higher topological type,
more precisely with two tunnels (see Figure \ref{fig:morgan} in the case
of the {\it regular} tetrahedron). 

\begin{figure}
\includegraphics[width=0.43\textwidth]{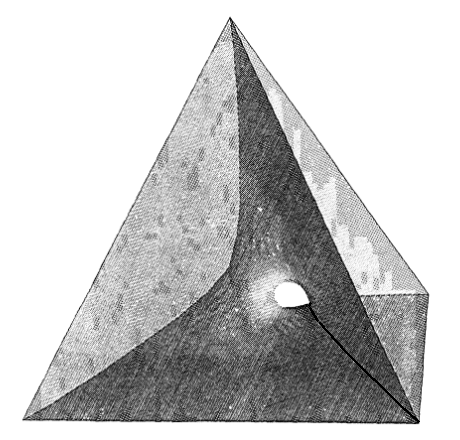}
\caption{\small{
A slightly retouched version of \cite[fig. 1.1.1]{LaMo:94}, see also
\cite[fig. 11.3.2]{Mo:08}. This soap film has two tunnels, one clearly
visible in the picture. This figure was done by Jean Taylor,
following an idea due to Bob Hardt.
}}
\label{fig:morgan}
\end{figure}

The cover described in \cite{BePaPa:17} has the particular feature of
being not normal; in addition, it is constructed using the above mentioned
disks. Similar disks
were  firstly introduced in \cite{Br:95} in other examples,
and called {\it invisible wires} by the author.
In the case of the tetrahedron, they play a crucial role. {}From
one side, they  are necessary to complete the construction of 
 the triple cover; from the other side, they act as an obstacle. In addition,
they allow one to distinguish {\it tight} loops around particular edges of the frame $S$  from
loops turning {\it far} from the edges: this distinction turns out
to be crucial for the modelization of a higher genus soap film.
 The 
results of  \cite{BePaPa:17} strongly suggest that,
for a tetrahedron sufficiently elongated in one direction, the higher-genus surface
has area strictly less then the conical configuration. 

In this paper, for convenience of
the reader we recall (Section \ref{sec:double_covers_of_R3_deprived_by_a_curve}) 
the double covers method and 
BV functions for treating  the classical
Plateau problem. In Section \ref{sec:covers_of_degree_larger_than_two} we point out 
the main modifications
of the construction in the case of covers of degree larger than two.
Next, in Section \ref{sec:examples}
 we continue the analysis in the spirit of \cite{BePaPa:17},
discussing various interesting examples. In Example \ref{exa:a_partially_wetted_curve}
we discuss with some care a classical example due to F.J. Almgren of a soap film only partially wetting
an unknotted curve, see also \cite{Br:95}. 
In Example \ref{exa:soap_film_on_a_cubical_frame}
we describe a cover of $\R^3 \setminus S$, where $S$ is the 
one-skeleton of a cube, which is compatible with 
the soap film depicted in Figure \ref{fig:nscube}. This is obviously
not the most common  soap film one usually finds in 
pictures, which has no holes and has triple curves starting in the corners
\cite[Figure 6]{Ta:76}. It is worthwhile to notice that such a soap film
has area larger than the area of the soap film in 
Figure \ref{fig:nscube}.
In Example \ref{exa:triple_moebius_band} we show how to construct
a triple cover compatible with the soap film of Figure \ref{fig:retract},
which is a surface that retracts on its boundary, and therefore for which 
we cannot apply the Reifenberg method. 
In Example \ref{exa:octahedron} we discuss the case when $S$ is the 
one-skeleton of an octahedron.

We conclude this introduction by mentioning that calibrations, applied to the covering
space method, have been considered in \cite{Br:95},
\cite{Br:95b} and, more recently,
in  \cite{CaPl:17} in connection with the
 BV approach in dimension two.

\section{Double covers of \texorpdfstring{$\Omega\setminus S$}{Omega - S}}
\label{sec:double_covers_of_R3_deprived_by_a_curve}
In this section we describe the cut and paste
method for constructing a double cover of the base space
$M:= \Omega\setminus S$ where, for simplicity, 
$S$ is a smooth
compact embedded  two-codimensional manifold without boundary
and $\Omega$
is a sufficiently large ball of $\R^n$ containing $S$, $n \geq 2$.
Just to fix ideas, one can consider $n=3$ and 
$S$  a tame knot or link\footnote{No
invisible wires will be taken into account in this section.}.
Next, to model the area minimization problem with $S$ as 
boundary datum, 
we 
define a minimum problem on a class of BV functions 
defined on the cover and satisfying a suitable constraint. The projection over the 
base space of 
the jump set of a minimizer will be our definition of solution to the Plateau problem;
this is a simplified
version of the construction described in \cite{AmBePa:17}, to which we refer for all details.
Before starting the 
discussion, it is  worth to recall that, in more general cases (such as
those in Section \ref{sec:examples}),
the cut and paste procedure needs not be 
the most convenient method to work with. 
Indeed, the cover can be equivalently described in two
other  ways. In the first one it is sufficient to declare
 an orientation 
of the cut, and a family of permutations of the strata along
the cut; this family must be consistent, a condition that is obtained 
from the local triviality of the cover. The second method is 
based on an 
abstract construction, by taking the quotient
of  the universal cover of $M$ with respect to a subgroup 
of the fundamental group of $M$; 
at the end of the section we 
recall this construction, while in Section \ref{sec:examples}
we shall use both these two latter methods.

In what follows
we shall always assume that the cover is trivial in a neighbourhood of $\partial \Omega$. Hence,
in that neighbourhood we can speak 
without ambiguities of  sheet one and sheet two,
up to automorphisms of the cover.

\subsection{Cut and paste construction of the double cover}
We start by defining 
a {\it cut} (also called a cutting surface when $n=3$), which is a $(n-1)$-dimensional
compact embedded smooth oriented submanifold  $\Sigma \subset \Omega$
with $\partial \Sigma = S$. 
Next we 
glue two copies (the sheets, or strata) of $M:= \Omega \setminus S$ 
along $\Sigma$ by exchanging the sheets.
Equivalently, we associate 
the permutation $(1 ~ 2)$ to $\Sigma$.\footnote{Note that, being this permutation
of order two, fixing  an orientation of $\Sigma$ is not necessary and
$\Sigma$ could even be nonorientable. For 
covers of degree larger than two and 
other type of permutations
(see Sections \ref{sec:covers_of_degree_larger_than_two} and \ref{sec:examples}) orientability of $\Sigma$
is necessary.}

To figure out the construction,
it is convenient to 
``double'' $\Sigma$, namely to slightly separate two copies of $\Sigma$
having boundary $S$ and meeting only at $S$; we call these two copies 
$\Sigma$ and $\Sigma'$, and
we denote by  $\mathbf \Sigma$ the pair $(\Sigma, \Sigma')$, that we call
pair of cuts.  
The orientability of $\cut$ 
gives a unit normal vector field on $\cut\setminus\Hole$\,---\,hence, in particular,
a direction to
follow in order to ``enlarge'' the cut, separating its two ``faces''.
If we call  $O\subset \Om$ (resp. $I \subset \Om$)
the open region exterior (resp. interior) to $\cut \cup \cut'$,
we can explicitely describe the gluing procedure as follows:
\begin{itemize}
\item[]
we let 
$$
\chart:=\Om \setminus \cut, \qquad \chart':=
\Om \setminus \cut',
$$
and consider\footnote{In order 
to be consistent with the permutation $(1~ 2)$ mentioned above,
it is sufficient  to rename $(D',3)$ and $(D',4)$ as $(D',1)$ and $(D',2)$.} 
$$
\unionedisgiunta:=(\chart,1)\cup (\chart, 2)\cup (\chart', 3)\cup (\chart', 4);
$$
\item[] we endow $\unionedisgiunta$ 
with the following  equivalence relation:
given $x, x' \in \mfd$ and
$j \in \{1,2\}$, $j' \in \{3,4\}$, 
$(x,j), (x',j')\in \mathcal X$, we say that 
$(x,j)$ is equivalent to $(x',j')$ if and only if
$x=x'$, and  one of the following
 conditions hold:

\begin{equation}\label{eq:equivalenza}
\begin{cases}
& x\in O, \qquad 
\{j,j'\}\in \big\{\{1,3\}, \{2,4\}\big\}, 
\\
& x\in I, \qquad 
~ \{j,j'\}\in \big\{\{1,4\}, \{2,3\}\big\}.
\end{cases}
\end{equation}
\end{itemize}
We call $\cvv$ the quotient space of $\unionedisgiunta$ by this equivalence relation (endowed with the quotient topology) and 
$\widetilde{\pi} : \unionedisgiunta \to \cvv$ the projection.
The double cover of $\mfd$ is then
\begin{equation}\label{eq:proiezionesullabase}
\pi_{\cutpair,\mfd}
\colon \cvv \to \mfd
\end{equation}
where $\pi_{\cutpair,M}(\widetilde{\pi}(x,j)):=x$ for any $(x,j)\in\unionedisgiunta$,
which is well defined, since if $(x,j)\sim(x',j')$, then
$\pi_{\cutpair,\mfd}(\widetilde{\pi}(x,j))=\pi_{\cutpair,\mfd}(\widetilde{\pi}(x',j'))$.
If we set $\pi\colon (x,j)\in \unionedisgiunta\mapsto x \in \mfd$,
we have the
following commutative diagram:
\begin{equation}\label{eq:schema}
\xymatrix{
\unionedisgiunta \arr[r]^{\widetilde{\pi}} \arr[dr]_{\pi}  & \cvv \arr[d]^{\pi_{\cutpair,M}} \\
& \mfd
}
\end{equation}
The quotient $\cvv$ 
 admits a natural
structure of differentiable manifold, with four
local parametrizations given by
$\Psi_1, \Psi_2, \Psi_3, \Psi_4$, where
\begin{equation}\label{eq:Psij}
\begin{aligned}
& \Psi_j \colon \chart \to \widetilde{\pi} \big((\chart,j)\big), \qquad
\Psi_j:=\widetilde{\pi}\circ \pi_{\vert_{(\chart,j)}}^{~~-1}, \quad j =1,2,
\\
& \Psi_{j'} \colon \chart' \to\widetilde{\pi} \big((\chart',j')\big), \qquad
\Psi_{j'}:=\widetilde{\pi}\circ \pi_{\vert_{(\chart',j')}}^{~~-1}, \quad j' =
3,4.
\end{aligned}
\end{equation}
It is important  here that the transition maps are the identity: 
$$\Psi_{j'}^{-1}\circ \Psi_j = \mathrm{id}=\Psi_j^{-1} \circ \Psi_{j'},
\qquad 
j\in\{1,2\}, j'\in \{3,4\},
$$
the equalities being valid where all members of the equation are defined.
Notice that 
$\Psi_1(D) \cup \Psi_2(D) = \cvv \setminus \pi_{\cutpair,M}^{\ -1}(\cut
\setminus S)$, and
$\Psi_3(D') \cup \Psi_4(D') = 
\cvv \setminus \pi_{\cutpair,M}^{\ -1}(\cut'
\setminus S)$.

The local parametrizations allow to read a function $u : \cvv \to \R$
 in charts:
for $j=1,2$ and  $j'=3,4$
we let
$\uv_j(u): D \to \R$,  $\uv_{j'}(u): D'\to \R$
 be 
\begin{equation}\label{eq:v}
\uv_j(u):=u\circ \Psi_j, \qquad
\uv_{j'}(u):=u\circ \Psi_{j'}.
\end{equation}
Recalling \eqref{eq:equivalenza}, we have
\begin{equation}\label{eq:vvvv}
\begin{split}
v_1(u)=v_3(u) , \quad & v_2(u)=v_4(u) \qquad{\rm a.e.~in~} O,\\
v_1(u)=v_4(u), \quad & v_2(u)=v_3(u) \qquad {\rm a.e.~in~} I.
\end{split}
\end{equation}

%

\subsection{Total variation on the double cover}
The set $\cvv$ 
 is endowed with
the push-forward $\mu$ 
of the $n$-dimensional Lebesgue measure $\mathcal L^n$ in $\mfd$
via the local parametrizations.
We set $L^1(\cvv) := L^1_\mu(\cvv)$.

We say that $u$  is in $BV(\cvv)$ if
its distributional
gradient
$D u \colon \phi\in C^1_c(\cvv) \mapsto -\int_{\cvv} u D \phi \,d\mu \in \R^n$
is a bounded vector\,--\,valued Radon measure on $\cvv$. We denote by $|Du|$ the
\emph{total variation} measure of $Du$.

Let $u \in BV(\cvv)$ and $E\subseteq \cvv$ be a Borel set;
$E$ can be written as the union
of the following four disjoint Borel sets:
\begin{equation}
\label{eq:splitting}
E \cap \pitilde((D,1)), \;
E \cap \pitilde((D,2)), \;
E \cap \pitilde((\cut\setminus S,3)), \;
E \cap \pitilde((\cut\setminus S,4)),
\end{equation}
and we have 
\begin{equation}\label{eq:vartotnuova}
\begin{aligned}
\vert Du\vert(E) =
&
\sum_{j=1,2} \vert D\uv_j(u)\vert \Big(\pi_{\cutpair,M}\big(E \cap \pitilde((D,j))\big)\Big)\\
& +
\sum_{j' =3,4} \vert D\uv_{j'}(u)\vert \Big(\pi_{\cutpair,M}\big(E \cap \pitilde((\cut\setminus S,j'))\big)\Big).
\end{aligned}
\end{equation}
Notice that $\cut'$ does not appear in \eqref{eq:splitting}.
Choosing $\chart'$ in place
of $\chart$ amounts in
considering $\cut'$ in place of $\cut$ and
does
not change the subsequent discussion.

\begin{Example}\label{exa:dischi}\rm
Suppose the simplest case $n=2$, and $S$ two distinct points $q_1, q_2$.
Let $u \in  BV(\cvv)$ be such that
$v_1(u)$ is equal to
$a \in \R$
inside a  disk $B$ of radius $r>0$ contained in $I$ (or in $O$)
and $b \in \R$
outside, and
 $v_2(u)$ is equal to
$c \in \R$ in $B$
and $d \in \R$
outside.
Then, owing to \eqref{eq:vvvv},
\begin{equation}\label{eq:esempiostupido}
\begin{split}
\vert Du\vert(\cvv) = &
\vert Dv_1(u)\vert(B \cap \chart) + \vert Dv_2(u)\vert(B \cap \chart)
\\
& +  \vert Dv_3(u)\vert (\cut \setminus \{q_1, q_2\})
+  \vert Dv_4(u)\vert (\cut\setminus \{q_1, q_2\})\\
=& (\vert b-a\vert + \vert d - c\vert) ~2 \pi r + 2
\mathcal H^1(\cut)
 \vert d-b \vert.
\end{split}
\end{equation}
On the other hand, if $B$
is centered at a point of $\cut$, and $B\cap
\cut' = \emptyset$, then
\begin{equation}\label{eq:esempio}
\begin{aligned}
\vert Du\vert(\cvv) = &
\vert Dv_1(u)\vert(B \cap \chart) + \vert Dv_2(u)\vert(B \cap \chart) \\
&
+ \vert Dv_3(u)\vert(\cut\setminus \{q_1, q_2\}) + \vert Dv_4(u) \vert(\cut\setminus \{q_1,q_2\})
\\
= &
\left(\vert b-a\vert + \vert d - c\vert\right) ~2 \pi r
+ 2 \vert c-a\vert \mathcal H^1(\cut \cap B) \\
& + 2 \vert d - b\vert \left(
\mathcal H^1(\cut)-\mathcal H^1(\cut \cap B)\right) .
\end{aligned}
\end{equation}
If in particular $a=1$, $b=0$, $c=0$, $d=1$,
we have that 
 \eqref{eq:esempiostupido} and  \eqref{eq:esempio}
become
\begin{equation*}\label{eq:utile}
\vert Du\vert(\cvv) = 2\left(2 \pi r + \mathcal H^1(\cut)\right).
\end{equation*}
\end{Example}

\subsection{The constrained minimum problem on the double cover}\label{subsec:setting}
We let
$$
BV(\cvv; \{0, 1\}) := \Big\{ u \in BV(\cvv) \; : \;   u(y) \in \{0,1\} {\rm ~for}~\mu~{\rm a.e.}~ y \in \cvv\Big\}.
$$
The domain of $\mathcal F$ is defined\footnote{For simplicity
we drop the dependence on $\cutpair$ in the notation.} by 
\begin{equation*}
\label{eq:vincolo_gen}
\BVc:=\Big\{u\in BV(\cvv; \{0, 1\}) \;  :  
\sum_{\pi_{\cutpair,M}(y)=x} u(y)=1 \text{~ for a.e.~$x$ in~} \mfd \Big\},
\end{equation*}
and 
$$
\mathcal F(u) := \vert Du\vert (\cvv), \qquad u \in \domain(\FFF).
$$

Therefore  the values
of $u$ on the two points of a fiber are $0$ and $1$: this is what we call
the {\it constraint on the fibers}.
Hence, for any $u \in \BVc$ we have
\begin{equation}
\label{eq:v1v2}
\uv_1(u)=1-\uv_2(u) \; \text{a.e.~in}~ \chart,
\qquad v_3(u)=1-v_4(u) \; \text{a.e.~in}~\chart'.
\end{equation}
For this reason, in formulas \eqref{eq:saltou} and \eqref{eq:formulafinale} below the functions
$v_2(u)$ and $v_4(u)$ are not present. Moreover, the following splitting 
formula holds:
\begin{equation}\label{eq:saltou}
\pi_{\cutpair,M}(J_u) = \Big( J_{\uv_1(u)} \setminus (\CutsmenoS) \Big) \cup
\Big( J_{v_3(u)}\cap (\Sigma \setminus S) \Big).
\end{equation}
Indeed, as in \eqref{eq:splitting}, let us split $J_u$ as the union of the following four disjoint sets:
\begin{equation}\label{eq:splitjump}
 J_u \cap \pitilde((\chart,1)) , \;  J_u \cap \pitilde((\chart,2)),
\;  J_u \cap \pitilde((\CutsmenoS,3)) , \;  J_u \cap \pitilde((\CutsmenoS,4)) .
\end{equation}
By the constraint on the fibers, to each point in the first
set of \eqref{eq:splitjump} there corresponds a unique point
in the second set, belonging to the same fiber, and viceversa.
A similar correspondence
holds between the third and the fourth set. Hence
\begin{equation*}\label{eq:splitjump2}
\pi_{\cutpair,M}(J_u) = \pi_{\cutpair,M}\Big( J_u \cap \pitilde((\chart,1)) \Big)
\cup \pi_{\cutpair,M}\Big( J_u \cap \pitilde((\CutsmenoS,3)) \Big).
\end{equation*}
By the definitions of $J_u$, $J_{v_1(u)}$ and  $J_{v_3(u)}$,
using also the local parametrizations $\Psi_1$,
$\Psi_3$, it follows  that
$\pi_{\cutpair,M}\big( J_u \cap \pitilde((\chart,1)) \big)= J_{\uv_1(u)}
\setminus (\CutsmenoS)$,
and $\pi_{\cutpair,M}\big( J_u \cap \pitilde((\CutsmenoS,3))\big)= J_{\uv_3(u)}
 \cap (\CutsmenoS)$, and
 \eqref{eq:saltou} follows.

\begin{definition}[\textbf{Constrained lifting}]\label{rem:uspeciale}
Let $v\in BV(\chart; \{0,  1\})$. Then
the function
\begin{equation}
\label{eq:usigma}
u:= \begin{cases}
v &  \text{in } \Psi_1(\chart),\\
1-v &  \text{in } \Psi_2(\chart),
\end{cases}
\end{equation}
is in $\BVc$, and $\uv_1(u)=v$. We call $u$
the constrained lifting of $v$.
\end{definition}
In particular,  when $v$ is identically equal to $1$ (or $0$),
we have
$$
\pi_{\cutpair,M}(J_u)= \CutsmenoS.
$$
The next result clarifies which is the notion of area we intend to minimize.

\begin{Proposition}
Let $u\in \BVc$. Then
\begin{equation}
\label{eq:formulafinale}
\begin{split}
|D u| (\cvv)
=&
2 \Big( \hn( J_{\uv_1(u)} \setminus \cut) + \hn(J_{v_3(u)} \cap \cut) 
\Big)
\\
=& 2 \,\hn(\pi_{\cutpair,M} (J_u)).
\end{split}
\end{equation}
\end{Proposition}
\begin{proof}
Recall the splitting in \eqref{eq:vartotnuova}, with the choice $E:=\cvv$.
By \eqref{eq:v1v2}, we have
\begin{equation}
\label{eq:contributo1}
|D \uv_1(u) | (\chart)=|D \uv_2(u) | (\chart), \qquad
|D \uv_3(u) | (\cut)=|D \uv_4(u) | (\cut).
\end{equation}
By the properties of BV functions we have
\begin{equation}
\label{eq:contributo2}
|D \uv_1(u) | (\chart) =  \hn( J_{\uv_1(u)} \setminus \cut), \quad
|D \uw_3(u)|(\cut)=
 \hn (J_{\uv_3(u)} \cap \cut).
\end{equation}
Substituting \eqref{eq:contributo2} into
\eqref{eq:vartotnuova}, and
recalling \eqref{eq:contributo1},
we get the first equality in \eqref{eq:formulafinale}.
The second equality  is now a consequence of \eqref{eq:saltou}.
\end{proof}

\begin{remark}\label{rem:fattore4}\rm
The factor $2$ in \eqref{eq:formulafinale} is obtained by
multiplying  the absolute value
of  the difference of the values of $u$
(which gives a factor $1$),
with the number of the sheets
(which gives a factor $2$).

\end{remark}
A particular case of a result proven in \cite{AmBePa:17} is 
the following. 
\begin{theorem}[\textbf{Existence of minimizers}]\label{teo:esistenza}
We have \begin{equation}
\label{eq:problema_gen}
\inf\Big\{|D u|\big(\cvv \big) \; : \; u\in \BVc \Big \} 
=
 \min\Big\{|D u|\big(\cvv \big) \; : \; u\in \BVc \Big \} >0.
\end{equation}
\end{theorem}
Positivity 
follows
from \eqref{eq:constancy2} below, with the choice
$A:=\Om$.
We
denote by $u_{\min}$ a minimizer of problem \eqref{eq:problema_gen}.

\begin{lemma}\label{lem:constancy}
Let $A\subseteq \Om$
be a nonempty open set such that $\pi_{\cutpair,M}^{\ -1}(A\setminus S)$ is connected.
Then for any $u\in\BVc$,
\begin{equation}
\label{eq:constancy}
\hn\big(A \cap \pi_{\cutpair,M}(J_u) \big)>0.
\end{equation}
Moreover, if $A$ is bounded with Lipschitz boundary,
then
\begin{equation}
\label{eq:constancy2}
\inf \big\{ \hn\big(A \cap \pi_{\cutpair,M}(J_u)  \big) \; : \; u\in\BVc \big\} >0.
\end{equation}
\end{lemma}
\begin{proof}
By contradiction,  suppose that
\begin{equation}
\label{eq:constancycontr}
\hn\big(A \cap \pi_{\cutpair,M}(J_u) \big)=0.
\end{equation}
Applying \eqref{eq:saltou} to
\eqref{eq:constancycontr}, we get
\begin{equation}\label{eq:contaccio2}
0 =
\hn(A  \cap (J_{v_1(u)} \setminus \Sigma))
+
\hn(A  \cap J_{v_3(u)} \cap  \Sigma).
\end{equation}
Now, set $A^S:=A\setminus S$.
Applying \eqref{eq:vartotnuova} with the choice
$E:=\pi_{\cutpair,M}^{-1}(A^S)$,
we get
\begin{equation}\label{eq:contaccio}
\begin{aligned}
\vert Du\vert(\pi_{\cutpair,M}^{\ -1}(A^S))
= &
2
\vert D v_1(u)\vert
\left(
\pi_{\cutpair,M}(\pi_{\cutpair,M}^{\ -1}(A^S) \cap\widetilde{\pi}((D,1))
)\right)
\\
&
+
2
\vert D v_3(u)\vert
\left(
\pi_{\cutpair,M}(\pi_{\cutpair,M}^{\ -1}(A^S) \cap \widetilde{\pi}(\Sigma \setminus S,3))
)\right)
\\
= &
2 \, \big(
\vert D v_1(u)\vert
\left(
A^S \setminus \Sigma
\right)
+
\vert D v_3(u)\vert
\left(
A^S \cap \Sigma
\right)
\big)\\
= & 2 \, \big(  \hn(A\cap (J_{v_1(u)} \setminus \cut)) + \hn(A\cap J_{v_3(u)} \cap \cut) \big),
\end{aligned}
\end{equation}
which, coupled with \eqref{eq:contaccio2}, implies
$|D u|(\pi_{\cutpair,M}^{\ -1}(A^S))=0$. Then
$u$ is constant on $\pi_{\cutpair,M}^{\ -1}(A^S)$, which contradicts the
validity of the constraint on the fibers.
This proves \eqref{eq:constancy}.

Now, let us  suppose,
still by contradiction, that there exists a sequence $(u_k)_k\subset \BVc$
such that $\lim_{k\to +\infty} \hn\big(A \cap \pi_{\cutpair,M}(J_{u_k})  \big) =0$. Thanks
to the assumption on $A$, $\pi_{\cutpair,M}^{\ -1}(A^S)$ is a double 
nontrivial cover of $A^S$.
In particular,
for each $k\in\mathbb N$, the restriction $\hat u_k:= {u_k}_{|_{\pi_{\cutpair,M}^{\ -1}(A^S)}}$ is
in $BV(\pi_{\cutpair,M}^{\ -1}(A^S); \{0,1 \})$ and satisfies the constraint on the 
fibers, and reasoning as above,
$|D {\hat u_k} |(\pi_{\cutpair,M}^{\ -1}(A^S))= 2 \hn(A \cap \pi_{\cutpair,M}(J_{u_k}))$.
By compactness, up to a not relabelled subsequence, 
there exists $\hat u\in BV_{\mathrm{constr}}(\pi_{\cutpair,M}^{\ -1}(A^S); \{0, 1 \})$ such
that $\hat u_k\to u$ in $L^1(\pi_{\cutpair,M}^{\ -1}(A^S))$, and  by lower semicontinuity,
\[
|D \hat u|(\pi_{\cutpair,M}^{\ -1}(A^S)) \leq  \liminf_{k\to+\infty} |D {\hat u_k} |(\pi_{\cutpair,M}^{\ -1}(A^S))
= 2 \lim_{k\to +\infty} \hn\big(A \cap \pi_{\cutpair,M}(J_{u_k}) \big) =0.
\]
Hence $\hat u$ is constant on $\pi_{\cutpair,M}^{\ -1}(A^S)$, contradicting the 
constraint on the fibers.
\end{proof}
Lemma \ref{lem:constancy} shows, in particular, that 
the nontrivial topology of the cover coupled with the
constraint on the fibers forces $u$ to jump in suitable
open sets.
As a further consequence of Lemma \ref{lem:constancy},
the boundary datum $S$ is attained 
 by any constrained function on the cover, in the 
following sense.
\begin{corollary}\label{cor:Sbordo}
Let $u\in\BVc$.
Then
\begin{equation}\label{eq:inclusione}
\overline{\pi_{\cutpair,M}(J_u)}\setminus \pi_{\cutpair,M}(J_u)
\supseteq S.
\end{equation}
\end{corollary}
\begin{proof}
The relation $S\cap \pi_{\cutpair,M}(J_u)=\emptyset$ is trivial, recall also \eqref{eq:saltou}.
Now,
suppose by contradiction that there exists a point
$p\in S\setminus \overline{\pi_{\cutpair,M}(J_u)}$.
Take an open ball $B$ centered at $p$, with $B\subset \Om\setminus \overline{\pi_{\cutpair,M}(J_u)}$,
and apply Lemma \ref{lem:constancy} with the choice $A:=B$. Then,
since $A \cap \pi_{\cutpair,M}(J_u)=\emptyset$, we end up with a contradiction
with \eqref{eq:constancy}.
\end{proof}
If $2 \leq n \leq 7$ 
and $u$ is a minimizer, it is possible 
to show that 
equality holds in \eqref{eq:inclusione} \cite{AmBePa:17}. 

The definition of solution to the Plateau problem in the 
sense of double covers\footnote{An analog definition can be given for covers of degree larger than two.
} is as follows.

\begin{definition}[\textbf{Constrained double\,--\,cover solutions}]\label{def:soluzione}
We call
$$
\pi_{\cutpair,M}(J_{u_{\rm min}})
$$
a \emph{constrained double\,--\,cover solution} (in $\Om$) to Plateau's problem \emph{with boundary $\Hole$}.
\end{definition}

We say that 
a portion $P$ of $S$ is wetted if $\overline{\pi_{\cutpair,M}(J_{u_{\rm min}})}
\supseteq P$, see also Section \ref{sec:examples}.

\subsection{Independence of the pair of cuts}\label{sec:indipendenza}
In this section we show that  constrained double\,--\,cover solutions are independent of  admissible cuts.
A  different proof of such an independence
is given in Proposition \ref{prop:isometria}.

Let us recall the   definition of unoriented linking number,
see for instance
\cite[Section 3.17]{BoTu:82} or \cite[Section 5.2]{Hi:76}.
\begin{definition}\label{def:link_property}
Let $\rho \in C^1(\Sf^1;\R^n \setminus \Hole)$
be transverse to $\cut$.
The unoriented linking number between $\rho$ and $\Hole$ is defined as
\begin{equation}\label{eq:link}
\lkn_2(\rho;\Hole):=
\begin{cases}
0 & {\rm if} ~ \#(\rho^{-1}(\cut)) {\rm ~is~even},
\\
1 & {\rm if} ~ \#(\rho^{-1}(\cut)) {\rm ~is~odd}.
\end{cases}
\end{equation}
\end{definition}

The right hand side of \eqref{eq:link} turns out to be independent of
the cut $\cut$.
When $\rho$ is just continuous, the unoriented linking number is defined
using a $C^1$ loop homotopic to $\rho$ and not intersecting $\Hole$
\cite{Hi:76}.

\medskip

\begin{theorem}\label{teo:indipendenza}
Let $\cutpair=(\cut,\cut')$, $\newcutpair=(\Gamma,\Gamma')$ be two pairs of cuts.
Let $u\in
BV(Y_{\cutpair}; \{0,1\})$ satisfies the constraint on the fibers.
Then there exists 
$u'\in BV(Y_{\newcutpair}; \{0,1\})$ satisfying the constraint
on the fibers 
such that, up to a $\hn$\,--\,negligible set,
\begin{equation}\label{eq:saltiuguali}
\pi_{\cutpair,\mfd}(J_u)= \pi_{\newcutpair,\mfd}(J_{u'}).
\end{equation}
\end{theorem}
\begin{proof}
Before giving the proof, we explain in a rough way the idea. 
First we fix a ``base point'' $x_0$ and count
the parity of 
the number of intersections with the various manifolds
$\Sigma, \Sigma', \Gamma, \Gamma'$.
Next, we construct $u'$ so that
 $u'$ coincides with $u$  when calculated on 
$(x,j)$ for $j=1,2$, provided
that the parity of the number of intersections with $\Sigma$ 
coincides with
the parity of the number of intersections with $\Gamma$, while
 $u'$ coincides with $1-u$  when calculated on 
$(x,j)$ for $j=1,2$, provided
that the parity of the number of intersections with $\Sigma$ 
differs with
the parity of the number of intersections with $\Gamma$. Similarly, 
 $u'$ coincides with $u$  when calculated on 
$(x,j')$ for $j'=3,4$, provided
that the parity of the number of intersections with $\Sigma'$ 
coincides with
the parity of the number of intersections with $\Gamma'$, while
 $u'$ coincides with $1-u$  when calculated on 
$(x,j')$ for $j'=3,4$, provided
that the parity of the number of intersections with $\Sigma'$ 
differs with
the parity of the number of intersections with $\Gamma'$.

Let us now come to the proof.
Without loss of generality, we can suppose that $\cut \neq \Gamma$.
Fix $\polo \in \mfd\setminus (\cut\cup\Gamma)$.
Let $x\in \mfd \setminus (\cut\cup \Gamma)$, and let $\gamma_x\in C^1([0,1];\mfd)$ be such that
$\ga_x(0)=\polo$, $\ga_x(1)=x$, and $\ga_x$ is transverse both to $\cut$ and to $\Gamma$; 
such a $\ga_x$ will be called an admissible path from $x_0$ to $x$.
We set
$$ 
{\mathit h}(\ga_x; \cut, \Gamma):= \#(\ga_x^{-1}(\cut)) + \#(\ga_x^{-1}(\Gamma)).
$$
If we consider
 another admissible path $\la_x$
from $x_0$ to $x$, we have 
that 
 $h(\ga_x; \cut, \Gamma)$ and  $h(\la_x; \cut,\Gamma)$
have the same parity.
Indeed,
let $\rho$ be the closed curve going from $\polo$ to $x$ following
 $\ga_x$, and
then backward from $x$ to $\polo$ along  $\la_x$.
Recalling that ${\rm link}_2(\rho; \cut) 
= {\rm link}_2(\rho; \Gamma)$, it follows that 
${\mathit h}(\ga_x;  \cut,\Gamma) + {\mathit h}(\la_x; \cut,\Gamma) 
=  \#(\rho^{-1}(\cut)) + \#(\rho^{-1}(\Gamma))$ is even.
We  are then allowed to  set
\begin{equation}
\label{eq:i2}
h(x; \cut,\Gamma):= 
\begin{cases}
0 & {\rm if} 
~{\mathit h}(\ga_x; \cut,\Gamma) ~{\rm is~ even}, 
\\
1 & {\rm if} 
~{\mathit h}(\ga_x; \cut,\Gamma) ~{\rm is~ odd}, 
\end{cases}
\end{equation}
for any admissible $\ga_x$ from $x_0$ to $x$\footnote{Once
$\polo$ is fixed, the function
$h$ allows 
to define an ``exterior'' and an ``interior'' 
of  $\cut\cup\Gamma$, even when $\cut$ and $\Gamma$
intersect on a set of positive $\hn$\,--\,measure.}.

Set $\mathcal Q 
:=\{x\in \mfd \setminus (\cut \cup \Gamma) \; : \; h(x; \cut, \Gamma)=0\}$,
which is an open set, 
with $\dde \mathcal Q \subseteq \cut \cup \Gamma$;
moreover 
$\mathcal Q$ has finite perimeter in $\Omega$
by
\cite[Proposition 3.62]{AmFuPa:00}.
Define
\begin{equation*}
\label{eq:v1primo}
v_1':=\begin{cases}
v_1(u) \quad & \text{in }  \mathcal Q,
\\
1-v_1(u) \quad & \text{in } \Om \setminus \mathcal Q.
\end{cases}
\end{equation*}
{}From 
\cite[Theorem 3.84]{AmFuPa:00} it follows that 
$v_1' \in BV(\Om; \{0,1\})$.
It also follows\footnote{
Indeed, let $x\in J_{v_1'} \setminus (\cut\cup\Gamma)$
and let $\ga_x$ be an admissible path from $x_0$ to $x$.
Let $B(x)$ be an open ball centered at $x$
and disjoint from $\cut\cup\Gamma$; in particular, every $z\in B(x)$ can be 
reached by a path 
obtained attaching to
$\ga_x$ the segment between $x$ and $z$; notice that
such a path $\ga_z$ is admissible from $x_0$ to $z$, and
${\mathit h}(\ga_z; \cut, \Gamma)={\mathit h}(\ga_x; \cut,\Gamma)$. 
Therefore,
either $v_1' = v_1(u)$ in $B(x)$  or $v_1'=1-v_1(u)$ in $B(x)$, which implies $x \in J_{v_1(u)}$. Hence
$J_{v_1'} \setminus (\cut\cup\Gamma) \subseteq
 J_{v_1(u)} \setminus (\cut\cup\Gamma)$. 
Similarly, also the converse inclusion holds, and \eqref{eq:primopezzosaltovprimo}
follows.
}
 that
\begin{equation}
	\label{eq:primopezzosaltovprimo}
	J_{v_1'} \setminus (\cut\cup\Gamma) = J_{v_1(u)} \setminus (\cut\cup\Gamma).
\end{equation}

We define  $u'\in BV_{\mathrm{constr}}(Y_{\newcutpair}; \{0,1\})$ as the 
constrained lifting 
of  $v_1'$
when $D$ is replaced by $\Omega \setminus \Gamma$.

Recalling also  \eqref{eq:vvvv},  set
$$
v_3' := 
\begin{cases}
v_1' & {\rm in~ the~ exterior~ region~} {\rm ~ to~} \Gamma\cup\Gamma',
\\
1-v_1' & {\rm in~ the~ interior~ region~} {\rm ~to ~} \Gamma\cup\Gamma'.
\end{cases}
$$
Notice that 
$v'_3\in BV(\Om; \{0,1\})$.
By construction, we have 
$$
v_1' = v_1(u'), \qquad 
v_3' = v_3(u').
$$
We claim
that $u'$ satisfies \eqref{eq:saltiuguali}. 
{}From \eqref{eq:saltou} we have
$$
\pi_{\newcutpair,\mfd}(J_{u'})= 
	\big ( J_{v_1'} \setminus (\Gamma \setminus S) \big) \cup 
\big( J_{v_3'} \cap (\Gamma \setminus S) \big),
$$
and our proof is concluded provided we show that, up to a $\hn$\,--\,negligible set,
\begin{equation}\label{eq:saltovvprimo}
	\big ( J_{v_1'} \setminus \Gamma \big) \cup 
\big( J_{v_3'} \cap \Gamma \big) =
\big(  J_{v_1(u)} \setminus \cut) \cup \big( J_{v_3(u)} \cap \cut\big).
\end{equation}
Let us split the left hand side of \eqref{eq:saltovvprimo} as
follows:
\begin{equation}
	\label{eq:saltovprimosplit}
	\begin{split}
	J_{v_1'} \setminus \Gamma &=  \Big( (J_{v_1'} \cap \cut) \setminus \Gamma\Big)
\cup \Big(J_{v_1'} \setminus (\cut\cup\Gamma)\Big),
\\
	J_{v_3'} \cap \Gamma &= \Big(J_{v_3'} \cap  \cut \cap \Gamma\Big)
	\cup \Big((J_{v_3'} \cap \Gamma)\setminus \cut\Big).
\end{split}
\end{equation}
Let us show that, up to a $\hn$\,--\,negligible
set,
\begin{equation}
	\label{eq:secondopezzosaltovprimo}
	(J_{v_1'} \cap \cut) \setminus \Gamma= (J_{v_3(u)} \cap \cut) \setminus \Gamma.
\end{equation}
Let $x\in (J_{v_1'} \cap \cut) \setminus \Gamma$. 
Up to a $\hn$\,--\,negligible
set\footnote{
Here we use again \cite[Theorem 3.84]{AmFuPa:00}.},
we can assume that the approximate tangent spaces to $J_{v_1'}$ and $\cut$ at $x$ coincide.
Let $B(x)$ be an open ball centered at $x$, not intersecting $\Gamma$, and such that
$B(x)\setminus \cut$ consists of two connected components. The same
argument used in the proof of \eqref{eq:primopezzosaltovprimo}  shows that on one component $v_1'=v_1(u)$, while
on the other $v_1'=1-v_1(u)$. Since $x\in J_{v_1'}$,
we have 
$$
x \notin  
J_{v_1(u)}. 
$$
On the other hand, by \eqref{eq:vvvv}, in one component we have
$v_1(u)=v_3(u)$, while in the other component $v_3(u)=v_2(u)=1-v_1(u)$ (where in the last equality
we used \eqref{eq:v1v2}). Thus, $x\in J_{v_3(u)}$. So, up to a $\hn$\,--\,negligible set,
$(J_{v_1'} \cap \cut) \setminus \Gamma \subseteq (J_{v_3(u)} \cap \cut) \setminus \Gamma$.
Arguing similarly for the other inclusion, we get
\eqref{eq:secondopezzosaltovprimo}.

The same argument applies also to prove that, up to a $\hn$\,--\,negligible set,
\begin{equation}
	\label{eq:terzopezzosaltovprimo}
	J_{v_3'} \cap  \cut \cap \Gamma= J_{v_3(u)} \cap  \cut \cap \Gamma, \quad
\end{equation}
and 
\begin{equation}\label{eq:quartopezzosaltovprimo}
	(J_{v_3'} \cap \Gamma)\setminus \cut =(J_{v_1(u)} \cap \Gamma)\setminus \cut. 
\end{equation}
From \eqref{eq:primopezzosaltovprimo}\,--\,\eqref{eq:quartopezzosaltovprimo},
we finally get \eqref{eq:saltovvprimo}.
\end{proof}
\begin{corollary}[\textbf{Independence}]\label{cor:indipendenza}
The minimal value in \eqref{eq:problema_gen} is independent of 
the pair $\cutpair$ of cuts.
\end{corollary}
\begin{proof}
 Let $\cutpair$, $\newcutpair$ be two pairs of cuts. Let
$u_{\min}\in \BVc$ be a function realizing the minimal value,
call it $\mathscr A(\cutpair)$.
Let $u'\in BV(Y_{\newcutpair}; \{0, 1\})$ be the function satisfying the constraint
on the fibers
given by
Theorem \ref{teo:indipendenza} applied with $u = u_{{\rm min}}$. Then, by \eqref{eq:formulafinale}
and \eqref{eq:saltiuguali}, we have
\[
\mathscr A(\newcutpair) 
\leq 2 \hn(\pi_{\newcutpair,\mfd}(J_{u'})) = 2 \hn(\pi_{\cutpair,\mfd} (J_{{u_{\min}}}))= \valoreminimocut.
\]
Arguing similarly for the converse inequality, we get
$\mathscr A(\newcutpair)=\valoreminimocut$.
\end{proof}
In view of Corollary \ref{cor:indipendenza}, we 
often skip the symbol $\cutpair$ 
in the notation of the cover, and on the minimal value of the area. 
Moreover, we often set 
$$
p := \pi_{\cutpair, \mfd}.
$$
\smallskip

The relations between a constrained double-cover
solution and other notions of solution to the Plateau problem
can be found in \cite{AmBePa:17}.

\subsection{Abstract construction of the double cover}\label{sec:abstract_construction_of_the_double_cover}
The construction of the abstract cover is standard
\cite{Ha:01}:
fix $\polo\in \mfda$, and set $\lacci:=\{\ga \in C\big([0,1]; \mfda\big) \;:\;  \ga(0)=\polo\}$. 
For $\ga\in\lacci$, let $[\ga]$ be the class of paths in $\lacci$ which are homotopic
to $\ga$ with fixed endpoints. We recall that the 
universal cover of $\mfda$
is the pair $(\ucov,\uproj)$, where
$\ucov:=\big\{[\ga] \ : \ \ga\in\lacci \big\}$ 
and $\uproj\colon [\ga]\in \ucov \mapsto \uproj([\ga]):= \ga(1)\in \mfda$.
The topology of $\ucov$ is defined as follows: consider the family
$\mathcal U:=\{B\subseteq \mfd \; : \; B \text{ open ball}\}$,
which is a basis of open sets of $\mfd$. For $B\in \mathcal U$, and for $[\ga]\in \ucov$ such that $\ga(1)\in B$, define
\begin{equation*}
U_{[\ga],B}:=\big\{ [\ga\la] \; : \; \la\in C([0,1]; B), \; \la(0)=\ga(1) \big\}.
\end{equation*}
Then a basis 
for the topology of $\ucov$ is given by
$\widetilde {\mathcal U} :=\{U_{[\ga],B} \; : \; B\in \mathcal U, \; [\ga]\in \ucov, \, \ga(1)\in B\}.$

Let $\piunoa$ be the 
fundamental group of $\mfda$ with base point $\polo$, and let
\begin{equation*}\label{eq:H}
H:=\{[\rho]\in \piuno \; : \; \lkn_2(\rho; \Hole)=0\},
\end{equation*}
which is a normal subgroup of $\piuno$ of 
index two.

For $\ga\in\lacci$, set $\bar\ga(t):=\ga(1-t)$ for all $t\in[0,1]$.
Associated with $H$, 
we can consider the following equivalence relation $\sim_H$ on $\ucov$:
for $[\ga], [\la] \in \ucov$, 
\begin{equation*}
[\ga]\sim_H [\la] \iff  \ga(1)=\la(1), \quad 
\lkn_2(\ga\bar\la; \Hole)=0.
\label{eq:Hrelbis}
\end{equation*}
\noindent We denote by $[\ga]_H$ the equivalence class of $[\ga]\in\ucov$ 
induced by $\sim_H$, and we set
\begin{equation*}\label{eq:abcv}
	\abcva:= \ucov / \sim_H.
\end{equation*}
Letting $\widetilde \projH_H\colon \ucov \to \abcva$ be the canonical 
projection 
induced by $\sim_H$, we endow $\abcva$ with the
corresponding quotient topology. We set
 $\projHM \colon [\ga]_H \in \abcva \mapsto \ga(1) \in \mfda$,
so that we have the following commutative diagram
\begin{equation}\label{eq:schema2}
\xymatrix{
\ucov \ar[r]^{\widetilde \projH_H} \ar[dr]_{\uproj}  & \abcva \ar[d]^{\projHM} \\
& \mfda
}
\end{equation}
and the pair $(\abcva,\projHM)$ is a cover of
$\mfda$, see \cite[Proposition 1.36]{Ha:01}.

Let $(Y,\pi_Y)$ be a cover of $\mfd$, and let $y_0\in\pi_Y^{-1}(\polo)$.
By $(\pi_Y)_*\colon \pi_1(Y,y_0)\to
\piunoa$ we denote the homomorphism defined as
$(\pi_Y)_*([\varrho]):= [\pi_Y\circ \varrho]$.
By \cite[Proposition 1.36]{Ha:01}, we have
\begin{equation}\label{eq:HHH}
(\projHM)_*(\pi_1(\abcv,[x_0]_H))=H,
\end{equation}
where $\pi_1(\abcv,[x_0]_H)$ is the fundamental group of $\abcv$ with base point
the equivalence class $[x_0]_H$ of the constant loop $x_0$.

\begin{Proposition}\label{pro:corollary}
Let $\mathbf{\cut}$ be a pair of cuts.
Then $\cvv$ and $\abcv$ are homeomorphic.
\end{Proposition}
\begin{proof}
By \cite[p.~28]{Ha:01}, we can assume that $\polo \notin \cut \cup \cut'$.
Now, let $y_0 \in \pi_{\cutpair,M}^{-1}(x_0)$ and 
 $[\varrho]\in \pi_1(\cvv, y_0)$. Then, $[\varrho]$ changes
sheet in $\cvv$ an even (or zero) number of times; therefore,
assuming without loss of generality $\varrho$ of class $C^1$ and transverse
to $\cut$,
recalling also \eqref{eq:link}, we have
\[
 0 \equiv \#\big((\pi_{\cutpair, M}\circ\varrho)^{-1}(\cut)\big) \equiv \lkn_2(\pi_{\cutpair,M}\circ\varrho; \Hole) \ \ \modtwo, 
\]
which implies $\pi_{\cutpair,M} \circ \varrho \in H$. Hence, $(\pi_{\cutpair,M})_*\big(\pi_1(\cvv, y_0)\big)\leq H$,
and since $H$ and $(\pi_{\cutpair,M})_*\big(\pi_1(\cvv, y_0)\big)$ have the same index,
they must coincide. {}From \eqref{eq:HHH}, we deduce
$$(\projHM)_*(\pi_1(\abcv,[x_0]_H))=(\pi_{\cutpair,M})_*\big(\pi_1(\cvv, y_0)\big).$$
By \cite[Proposition 1.37]{Ha:01}, the proof is complete.
\end{proof}
The homeomorphism between the two covers, which we denote 
\begin{equation}\label{eq:omeo}
\isom \colon \abcv \to \cvv,
\end{equation}
is given for instance in the proof of
\cite[Proposition 1.33]{Ha:01}: for $[\ga]_H \in\abcv$, let $\beta\in C([0,1];\abcv)$
be a path from $[\polo]_H$ to $[\ga]_H$; we uniquely
lift $\projHM\circ \beta$ 
to a path in $\cvv$ with base point $y_0$.
Then, $\isom([\ga]_H)$ is defined as
the endpoint of the lifted path, which turns out to be independent of $\beta$.

Let us define the distance $\dabcv$ on $\abcv$ 
 as follows: for $[\ga]_H,$ $[{\la}]_H \in \abcv$,
\begin{equation}\label{eq:dabcv}
\dabcv([\ga]_H,[\la]_H):=\inf_\beta \sup  \big\{ \sum_l  | \projHM(\beta(t_l))-
\projHM(\beta(t_{l-1})) | \; : \; (t_l)_l\in \mathrm{Part}(\beta) \big\},
\end{equation}
where the infimum runs among all $\beta\in C([0,1]; \abcv)$ connecting $[\ga]_H$ and $[\la]_H$;
for any such $\beta$, $\mathrm{Part}(\beta)$ denotes the collection of all finite partitions
$(t_l)_l$ of $[0,1]$ such that, for every $l$, 
there exist $[\ga_l] \in \widetilde M$
and a ball $B_l\subseteq M$ with $U_{[\gamma_l],B_l}\in \widetilde{\mathcal U}$
such that $\beta([t_{l-1},t_l]) \subset \widetilde \projH_H(U_{[\gamma_l],B_l})$. 

Symmetry, positivity, and the triangular inequality of $\dabcv$
are direct consequences of the definition.
Let us show that $\dabcv([\ga]_H,[{\la}]_H)=0$ implies $[\ga]_H=[{\la}]_H$.
Clearly, we have $\ga(1)={\la}(1)$. 
Fix $\eps>0$, and let
$\beta\in C([0,1], \abcv)$, $N\in \mathbb N$, $(t_l)_l\in \mathrm{Part}(\beta)$, $l\in \{1,\dots, N\}$,
be such that $\sum_{l=1}^N |\projHM(\beta(t_l))-\projHM(\beta(t_{l-1}))| \leq \eps$.
In particular, for $\eps>0$ sufficiently small, the closed curve $\rho$ defined as\footnote{
Here by $[\![x, x']\!]$ we mean the path corresponding to the segment
from $x$ to $x'$, for every $x,\,x'\in \mfd$.}
$$\rho:=[\![\ga(1), \projHM(\beta(t_1))]\!]
\cdots
[\![\projHM(\beta(t_{N-1})), \la(1)]\!]$$ 
is contractible in $\mfd$, which implies that
\begin{equation}\label{eq:linkrho}
\lkn_2(\rho;\Hole)= 0.
\end{equation}
By definition of ${\rm Part}(\beta)$, for every $l\in \{1,\dots, N\}$ there exist $\la_{l,1}$, $\la_{l,2} \in C([0,1]; B_l)$, with
$\la_{l,1}(0)=\la_{l,2}(0)=\ga_l(1)$, and such that $\beta(t_{l-1})=[\ga_l\la_{l,1}]_H$,
$\beta(t_l)=[\ga_l\la_{l,2}]_H$; notice that, since
$[\ga_{l-1} \la_{l-1,2}]_H=\beta(t_{l-1})=[\ga_{l}\la_{l,1}]_H$, we have
\begin{equation}
\label{eq:betall}
\lkn_2 (\ga_{l-1}\la_{l-1,2}\bar\la_{l,1}\bar \ga_l; \Hole)=0.
\end{equation}
Set
$
\rho_l:= \ga_l \la_{l,1} [\![\la_{l,1}(1), \la_{l,2}(1)]\!] \bar\la _{l,2}\bar\ga_l,
$
which is a closed curve in $\mfd$. In particular,
\begin{equation}
\label{eq:contraibile}
\lkn_2(\rho_l; \Hole) = \lkn_2(\la_{l,1} [\![\la_{l,1}(1), \la_{l,2}(1)]\!] \bar\la_{l,2}; \Hole)= 0,
\end{equation}
where last equality follows recalling that $B_l$ is contractible in $\mfd$.

Coupling \eqref{eq:linkrho}, \eqref{eq:betall} and \eqref{eq:contraibile}, we get
\begin{equation*}
\begin{split}
\lkn_2(\ga{\bar\la};\Hole) = & \lkn_2(\ga_0 \la_{0,1} \bar\la_{N,2} \bar\la; S)\\
=& \sum_{l=1}^N \Big( \lkn_2 (\rho_l; \Hole)  +
\lkn_2 (\ga_{l-1}\la_{l-1,2}\bar\la_{l,1}\bar\ga_l; \Hole)  \Big)
+ \lkn_2 (\rho; \Hole) =0 .
\end{split}
\end{equation*}
Hence $[\ga] \sim_H {[\la]}$, and 
the conclusion follows.

Now, we are in the position to establish the 
isometry bewteen the two covers.
We endow $\cvv$ with the distance $\dcv$ defined as follows:
for any $y$, $y' \in \cvv$, we set
\begin{equation}
\label{eq:dcvv}
\dcv\big( y, y'\big)=\inf _\eta\
\sup \big\{\sum_l |\pi_{{\bf \Sigma},M}(\eta(t_l) )- \pi_{{\bf \Sigma},M}(\eta(t_{l-1}))|  \ : \ (t_l)_l\in \Part(\eta) \big \},
\end{equation}
where the infimum runs among all $\eta \in C([0,1]; \cvv)$ connecting $y$ and
$y'$, and $\Part(\eta)$ is the family of all finite partitions $(t_l)_l$ 
of $[0,1]$
such that, for every $l$,
$\eta([t_{l-1}, t_l])$ is contained in a single chart of $\cvv$.
\begin{Proposition}[\textbf{Isometry}]\label{prop:isometria}
The map $\isom$ in \eqref{eq:omeo} is an isometry
between $(\abcv,\dabcv)$ and $(\cvv,\dcv)$.
\end{Proposition}
\begin{proof}
Let $[\ga]_H$, $[\la]_H\in \abcv$. For $\eps>0$, let $\beta\in C([0,1]; \abcv)$
be a path from $[\ga]_H$ to $[\la]_H$, realizing
the infimum in \eqref{eq:dabcv} up to a contribution of order $\eps$. Now, set
$\eta:=\isom\circ\beta$; accordingly to \eqref{eq:dcvv}, let $(t_l)_l\in \mathrm{Part}(\eta)$
be such that
$$\dcv(\isom([\ga]_H),\isom([\la]_H)) \leq \sum_l |\pi_{{\bf \Sigma},M}(\eta(t_l)) - \pi_{{\bf \Sigma},M}(\eta(t_{l-1}))| +\eps.$$
Clearly, it is not restrictive to assume that, for every $l$, $\pi_{\cutpair,M}(\eta([t_{l-1},t_l]))\subset B_l$, for some
open ball $B_l\subset \mfd$. Therefore, accordingly to
\eqref{eq:dabcv}, we have $(t_l)_l\in \mathrm{Part}(\beta)$; hence, for every $l$,
\[
|\pi_{\cutpair,M}(\eta(t_l) )- \pi_{\cutpair,M}(\eta(t_{l-1}))|=|\projHM(\beta(t_l)) - \projHM(\beta(t_{l-1}))|,
\]
which implies
\[
\dcv(\isom([\ga]_H),\isom([\la]_H))
\leq \dabcv([\ga]_H,[\la]_H) +2\eps.
\]
By the arbitrariness of $\eps$, we get
$\dcv(\isom([\ga]_H),\isom([\la]_H))\leq \dabcv([\ga]_H,[\la]_H)$.
Similarly, we get the converse inequality.
\end{proof}

Once we have to minimize a functional defined on some functional domain, the 
metric structure (and not only its topology) of the cover becomes relevant: 
the distance function on $Y$ is locally euclidean, and the two methods
described above give isometric covers.

We conclude this section remarking that a large part of what we have described  can be 
generalized \cite{AmBePa:17}:

\begin{itemize}
\item[$\bullet$] 
 to a cover of $\R^n \setminus S$ having 
more then two sheets. Allowing three or more sheets has the 
interesting by-product of modelling singularities in soap
films such as triple junctions (in the plane),
or triple curves (in space), quadruple points, etc.  
\item[$\bullet$] when $S$ is not smooth, 
for instance $S$ the one-skeleton of a polyhedron.
\end{itemize}
We refer to 
\cite{Br:95}, \cite{AmBePa:17} and \cite{BePaPa:17} for a more complete
description for covers of any (finite) degree.

\section{Covers of degree larger than two}\label{sec:covers_of_degree_larger_than_two}
The use of  covers $p := \pi_{\cutpair,M}: Y \to M$ of degree larger than two, coupled with 
vector-valued  BV-functions defined on $Y$ and satisfying a suitable constraint,  
is of interest since for instance:
\begin{itemize}
\item[$\bullet$] when $n=2$, one can  model, among others,  the Steiner
minimal graph problem connecting a finite number $k \geq 3$ of points in 
the plane \cite{AmBePa:17}; 
\item[$\bullet$] when $n=3$, one can consider configurations with singularities (triple curves,
quadruple points etc.),
in particular when $S$ is the one-dimensional skeleton
of a polyhedron;
\item[$\bullet$] choosing carefully the cover,
it is possible to model soap films with higher topological genus, as in the 
example of the one-skeleton of a tetrahedron\footnote{The triple cover
constructed in \cite{BePaPa:17} used to realize a soap 
film with two tunnels is not normal. Roughly, this means that one
of the three sheets is treated in a special way; this is also related
to the Dirichlet condition imposed on the cover in correspondence of the 
boundary of $\Omega$.}
 discussed in  \cite{BePaPa:17}: the resulting soap film seems
not to be modelable  using the Reifenberg approach \cite{Re:60}.
\end{itemize}

Some remarks to be pointed out
are the following:
\begin{itemize}
\item[$\bullet$] in the construction of the cover, and to model interesting
situations, 
it frequently happens to make use\footnote{Invisible wires
can be useful also for covers of degree two.} of what the author of \cite{Br:95}
called ``invisible wires'':
these may have various applications,
such as making globally compatible the cover, or also acting as an obstacle
(see also Section \ref{sec:examples}).
 They are called invisible wires because
the soap film should be supposed to wet the initial wireframe
$S$, but not to wet the invisible wires, 
so that their actual position becomes relevant.
Proving that a soap film has no convenience to wet the invisible wires
for special choices of their position, seems to be an 
open problem, not discussed in \cite{Br:95}. 
We refer to \cite{BePaPa:17} for more. 
\item[$\bullet$] Instead of describing explicitely
the cut and past procedure (as in Section \ref{sec:double_covers_of_R3_deprived_by_a_curve})  and the parametrizing maps
(which becomes more and more complicated as the degree of the 
cover increases)  now it is often convenient
to construct the cover first by orienting 
all portions\footnote{It is worth noticing that
it may happen that now the cut surface is immersed, and 
not embedded.} of
the cut,  then  declaring in a consistent global way
the  permutations for gluing the sheets along the  cut, and 
finally to use the local triviality of the cover, in order to check
the consistency of the gluing. Already in the case of triple covers, 
a relevant fact is the use of permutations with fixed points.

\item[$\bullet$]
Another useful way to describe the cover is 
the  abstract construction (already considered in Section 
\ref{sec:abstract_construction_of_the_double_cover} for double covers):
one has to 
suitably quotient
the universal cover with a subgroup of the fundamental group 
of the complement of $S$\footnote{or, if necessary, of the union of $S$
and the invisible wires.}. A clear advantage
of this approach is its independence of any cut, a fact that, with the
cut and past procedure, requires a proof.
\item[$\bullet$] BV-functions defined on $Y$ could be vector valued, as
in \cite{AmBePa:17}.
Suppose for simplicity to consider a triple cover; then one
choice is to work with BV-functions $u : Y \to \{\alpha, \beta, \gamma\}$, where $\alpha,\beta,\gamma$
are the vertices of an equilateral triangle of $\R^2$, having its barycenter at the origin. If $x$
is any point of $M$ 
and $p^{-1}(x) = \{y_1,y_2,y_3\}$ is the fiber over $x$, then we require $\{u(y_1), u(y_2), u(y_3)\}
= \{\alpha, \beta,\gamma\}$. Clearly, the constraint  now reads as  $\sum_{i=1}^3 u(y_i) =0$.

\noindent Another choice (made also in \cite{BePaPa:17})
is, instead, the following.
Again, suppose for simplicity to consider a triple cover. We can consider
BV-functions  $u : Y \to \{0,1\}$, so that 
if $x$ is any point of $M$ 
and $p^{-1}(x) = \{y_1,y_2,y_3\}$ is the fiber over $x$, then we require 
the constraint
$\sum_{i=1}^3 u(y_i) =1$. Other choices of the constraint are conceivable, but 
we do not want to pursue this issue in the present paper.
\end{itemize}

Once we have specified the domain of the area functional, i.e., a class of 
constrained BV-functions $u$, the variational 
problem becomes, as in Section \ref{sec:double_covers_of_R3_deprived_by_a_curve}, to minimize the total variation of
$u$\footnote{In the case of $u(y) \in \{\alpha,\beta,\gamma\}$,
the total variation is 
using the Frobenius norm $\vert 
T\vert = \sqrt{\sum (t_{ij})^2}$ on 
matrices $T = (t_{ij})$.}. This turns out to be the $(n-1)$-dimensional Hausdorff measure of 
the projection $p(J_u)$ of the jump set $J_u$ of $u$, times a positive
constant $c$, 
related to the codomain of $u$ and possibly to 
the number of sheets. For instance, for $u(y) \in \{\alpha,\beta,\gamma\}$ as above,
then $c=3 \ell$, where $\ell = \vert \beta-\alpha\vert$. 
For $u(y) \in \{0,1\}$, then $c=2$.

In the next section we construct triple covers, in some interesting
cases not considered in \cite{BePaPa:17}, and only partially considered 
in \cite{Br:95}.
\section{Examples}\label{sec:examples} 
In this section all covers are of  degree three; moreover, 
we consider BV functions $u : Y \to \{0,1\}$
with the constraint that the sum of the values of $u$ on the three points
of each fiber equals $1$.

We start with the example of Figure \ref{fig:cravatta}, due to F.J. Almgren
\cite[Fig. 1.9]{Al:01}.

\begin{figure}
\includegraphics[width=0.48\textwidth]{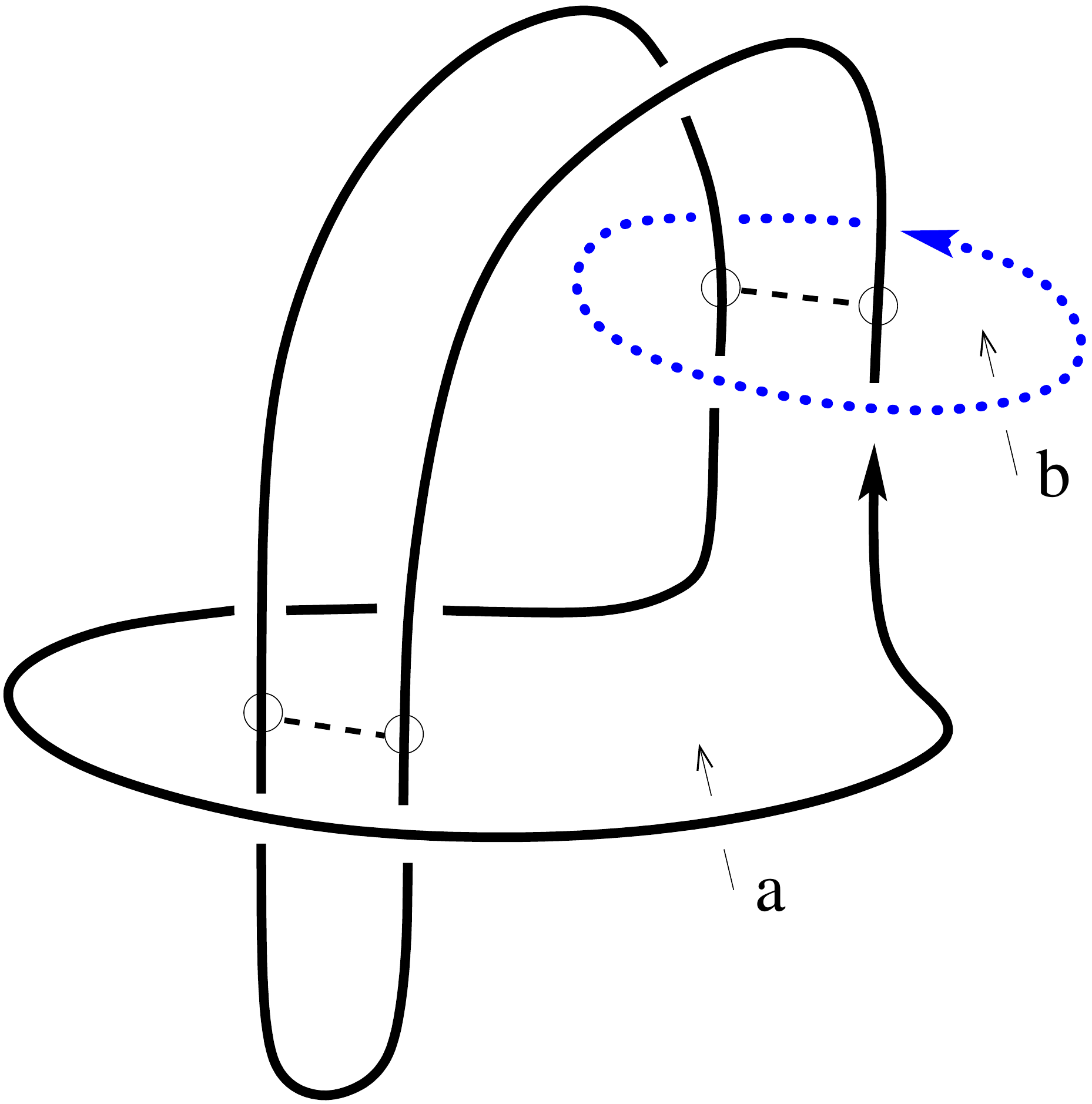}
\includegraphics[width=0.48\textwidth]{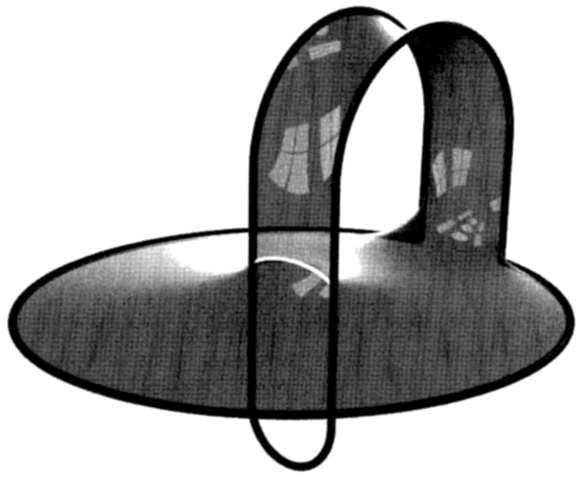}
\caption{\small{Left: an unknotted boundary (bold curve).
The dotted loop represents an invisible wire that is not part of the problem but
essential for the cover construction.
Right: a striking example of minimal film that only partially touches the boundary, due to
Almgren \cite[fig. 1.9]{Al:01}.}}
\label{fig:cravatta}
\end{figure}

\begin{Example}[\textbf{A partially wetted curve}]\label{exa:a_partially_wetted_curve}\rm
Let $\edges$ be the (unknotted) bold curve
in Figure \ref{fig:cravatta} (left). We want to construct a cover of $\R^3
\setminus \edges$ compatible with the soap film in Figure \ref{fig:cravatta} (right),
where the lower part is not wetted.
The presence of the 
triple curve suggests to use a cover of degree at least
three, and indeed three will suffice.
Removal of the unknotted curve from $\R^3$ leaves a set with infinite cyclic fundamental
group (isomorphic to $\Z$).

The only possible cover with three sheets that can be constructed on such a base space would
necessarily imply a cyclic permutation of the three points of the fiber when looping around
the lower portion of the curve, forcing an undesired wetting.
Similarly to the construction described in \cite{BePaPa:17} and in the same spirit as in
many of the examples in \cite{Br:95},  we then add an ``invisible wire'' in the form of a
loop circling the pair of nearby portions of $\edges$ in the upper part.
This is represented by the dotted loop $\iwire$ in Figure \ref{fig:cravatta} (left).
The base space $M$ is then defined as $\R^3 \setminus (\edges \cup \iwire)$.

A cut and past construction of the cover $\prj : \cover \to \base$ can now be defined
by cutting $\base$ along two surfaces bounded by $\edges$ and by $\iwire$ respectively.
The first one resembles the film of Figure
\ref{fig:cravatta} (right), but it has a selfintersection along the dashed (lower) segment and continues
below the disk-like portion touching the whole of $\edges$;
the second one is a small disk bounded by $\iwire$, intersecting the first cutting surface along
the dashed segment.
We now take three copies, numbered $1, 2, 3$, of the cutted version of $\base$ and glue them
along the cutting surfaces according to given permutations of the three sheets,
that we now describe.

The permutation along the lower portion of $\edges$ is chosen as $(2~3)$, namely stratum 1
glues with itself, while strata 2 and 3 get exchanged.
This choice is justified because we do not want to force wetting of that portion, indeed
a function in $\domainF$ defined equal to $1$ in sheet $1$ does not jump along a tight
loop around that part of $\edges$.
This choice in turn requires that we fix the Dirichlet-type condition $u=1$ 
out of a sufficiently large ball on stratum $1$ of the cover.

The permutations on the remaining parts of the cut can then be chosen consistently as
follows:
\begin{itemize}
\item[$(2~3)$] (as already described) in the lower tongue-like portion of the surface bordered
by $\edges$;
\item[$(2~3)$] when crossing the disk-like surface bordered by $\iwire$;
\item[$(1~2)$] when crossing the large disk-like portion of the surface bordered by
$\edges$;
\item[$(1~3)$] when crossing the ribbon-like portion of the surfaces between the two
dashed crossing curves.
\end{itemize}

Note that corresponding to portions of the surface that are wetting the 
bold curve, stratum $1$ is exchanged with a different stratum.

It is a direct check that with this definition the local triviality
of the triple cover around the triple curves,
namely that a small loop around the dashed curves must be contractible in $M$,
is satisfied.
This check consists in showing that
the composition of the three permutations associated with the 
crossings must produce the identity:
$(2~3) (1~2)^{-1} (1~3)^{-1} (1~2) = \id$.
The construction is actually unique up to exchange of sheets $2$ and $3$.


The fundamental group $\pi_1(\base)$  of $\base$ is readily seen to be free of rank $2$.
It can be generated by the two Wirtinger generators schematically denoted by
$a$ and $b$ in Figure \ref{fig:cravatta} left.
We can then finitely present $\pione$ with two generators and no relation as
$$
\pione = <a, b;> .
$$
An abstract construction of the cover can be obtained by considering the homomorphism
$\varphi : \pione \to \SSS_3$ (permutations of the set $\{1,2,3\}$)
defined by the position $\varphi(a) = (1~2)$, $\varphi(b) = (2~3)$
and then defining the subgroup $H < \pione$ as
$$
H = \{ w \in \pione : \varphi(w): 1 \mapsto 1 \} .
$$
It contains all reduced words $w \in \pione$ whose image under $\varphi$
is either the identity $\id \in \SSS_3$ or the transposition $(2~3)$.
It is a direct check that $H$ has index $3$ in $\pione$ and that it is
not normal.

As discussed in \cite{BePaPa:17} for the example of the tetrahedral wire,
also in this example we cannot exclude a priori that a minimizing surface
wets the invisible wire: we have already remarked that
this is a difficulty present in any example
constructed using invisible wires.

Finally, we recall
that soap films that partially wet any knotted curve 
have been proven to  exist in \cite{Pa:92}.
\end{Example}

The soap film of the next  example can be found for instance in \cite[pag. 85 and 
Fig. 4.14]{Is:92}.

\begin{Example}[\textbf{Soap film with triple curves 
on a cubical frame}]\label{exa:soap_film_on_a_cubical_frame}\rm
Let $S$ be the one-dimensional skeleton of the cube (Figure \ref{fig:nscube}).
We want to construct a cover of $\base =\R^3
\setminus S$ which is compatible with the soap film in Figure \ref{fig:nscube}; 
note that here the soap film wets all the edges of the skeleton.

\begin{figure}
\includegraphics[width=0.95\textwidth]{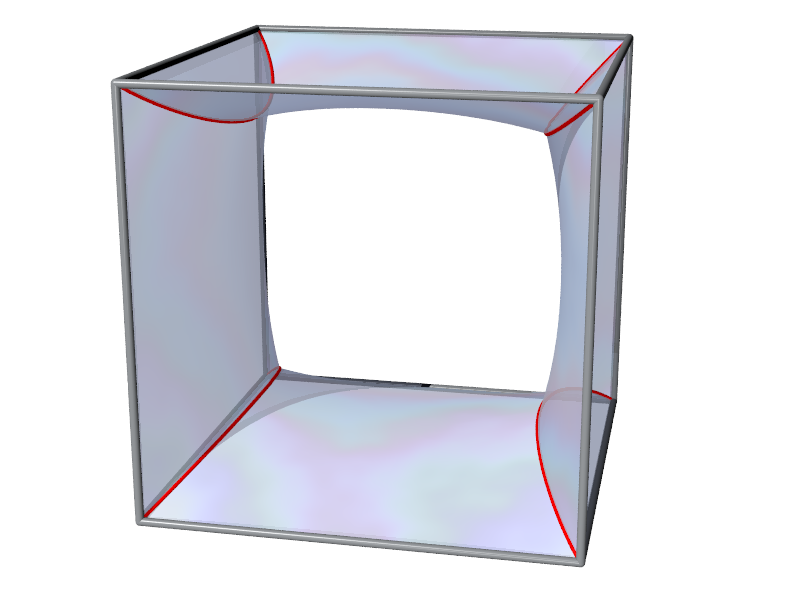}
\caption{\small{A non-simply connected minimal film spanning a cube.
Image obtained using the \texttt{surf} code by E. Paolini.}}
\label{fig:nscube}
\end{figure}

\begin{figure}
\includegraphics[width=0.60\textwidth]{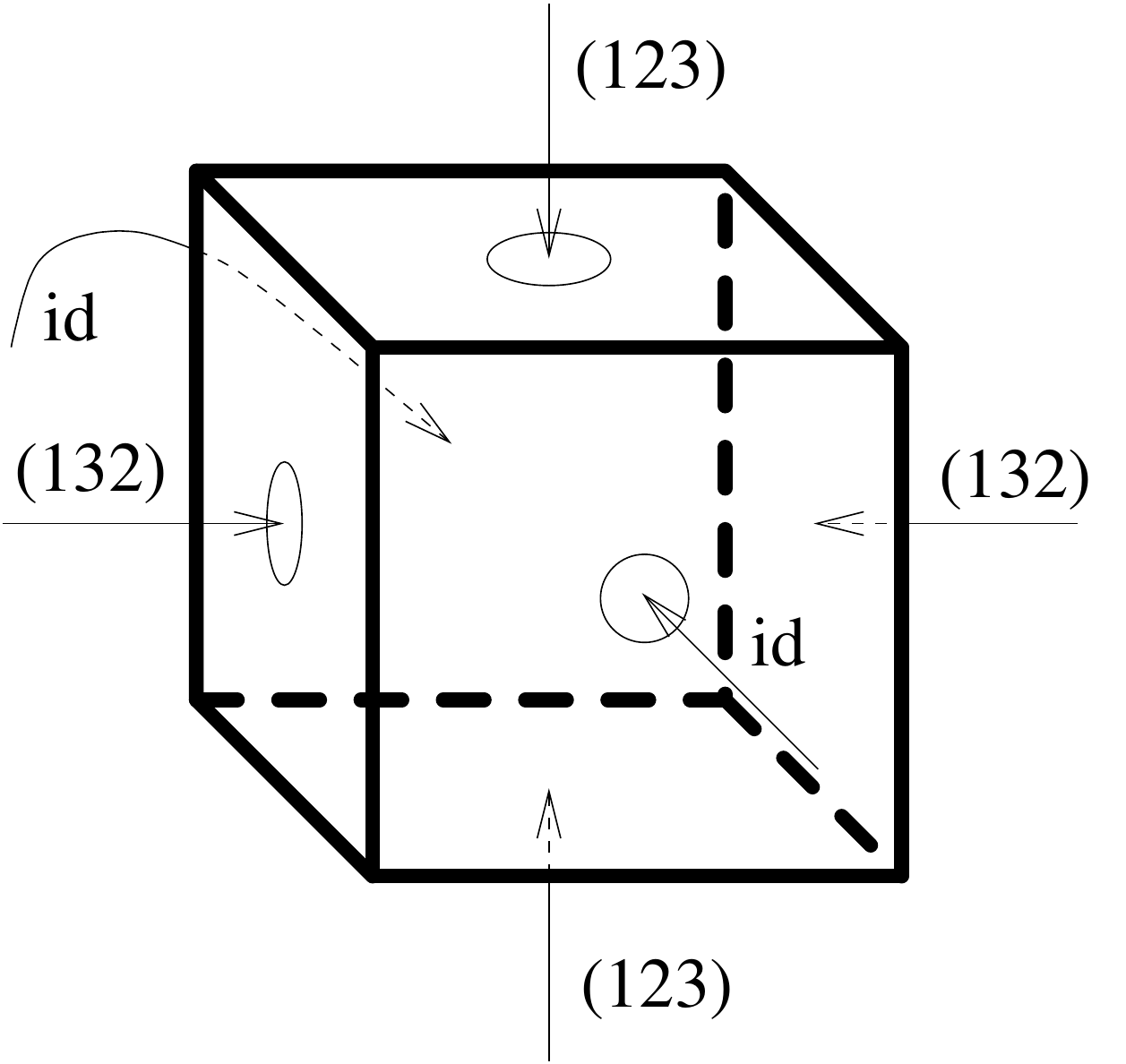}
\includegraphics[width=0.39\textwidth]{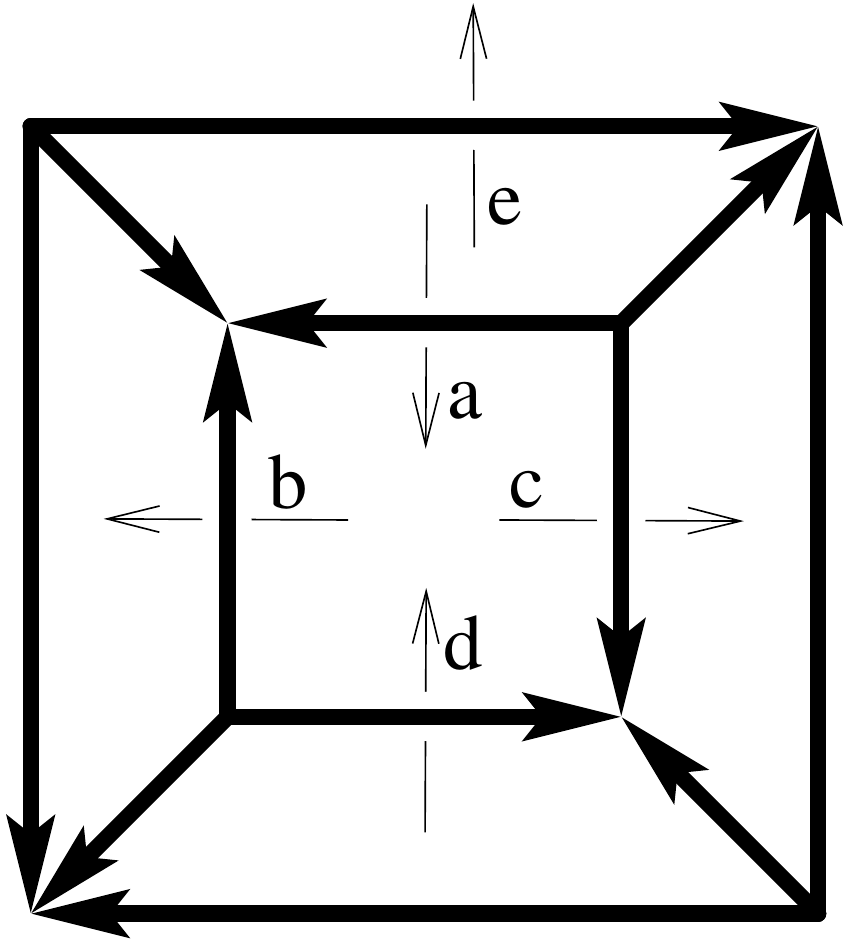}
\caption{\small{Left: orientation of the cut (the faces
of the cube), and permutations of the sheets along the cut.
 Right: the Wirtinger presentation of the fundamental group of the complement
of the one-skeleton of a cube. 
}}
\label{fig:cube}
\end{figure}

Again, we want to model a soap film with triple curves,
but not with quadruple points, and indeed, as we shall see, 
 a triple cover of $\base$ will suffice. Also, there will be no need of any invisible wire.
First of all, we orient the three pairs of opposite faces 
of the cube from 
the exterior to the interior,  as in Figure \ref{fig:cube} (left).
It turns out that we can make use of the cyclic permutations of $\{1,2,3\}$. 
We imagine a cut along the six faces of the cube, and we associate 
the same permutation to opposite faces: the identity permutation $\id$ 
is associated to the 
frontal and back
faces, in order to model the presence of the tunnel. The three powers 
$\id, (1~2~3), (1~3~2)$ 
of the cyclic permutation $(1~2~3)$ are depicted in Figure \ref{fig:cube}.
The presence of the identity permutation on a pair of opposite faces
has the effect of actually not having a cut there. On the other hand, 
a tight loop around an edge turns out in the composition of a power of 
$(1~2~3)$ with the inverse of a different power of 
 $(1~2~3)$, so that the result is either
$(1~2~3)$ or $(1~3~2)$, hence a permutation without 
fixed points, which forces to wet that edge.

Observe that a curve entering a face and exiting from the opposite one
produces the identical permutation of the strata of the cover, hence
it does not necessarily has to meet the projection of the jump set of 
a function $u$.

The fundamental group of $\base$ turns out to be a free group of rank $5$, 
and it can be generated
by the elements of $\pione$ schematically displayed in Figure \ref{fig:cube} 
(right) as $a$, $b$, $c$, $d$, $e$;
the corresponding Wirtinger presentation is
$$
\pione = <a,b,c,d,e;>
$$
(five generators and no relations).
Observe that the orientation of the edges in the figure is chosen such that all five generators loop positively around
the corresponding edge and result 
in the permutation $(1~2~3)$ of the three sheets when compared with the cut/paste
construction.
This allows an abstract definition of the cover by considering the homomorphism $\varphi : \pione \to \SSS_3$
that maps all five generators onto the cyclic permutatioon $(1~2~3)$ and take the normal subgroup $H < \pione$,
kernel of $\varphi$.
A word $w \in \pione$ belongs to $H$ whenever the exponent sum with respect to all generators is a multiple of $3$.

The abstract construction shows that this cover is normal.
Note that this construction is invariant (up to isomorphisms) under the symmetry group of the cube, hence
a minimizer will not be unique unless it is invariant under such symmetry group, which we do not expect to be true
in view of the film displayed in Figure \ref{fig:nscube}.

Minimizers with this topology were also obtained by real experiments \cite{Is:92}.

\end{Example}

The next example (Figure \ref{fig:retract}, found by J.F. Adams in \cite[Appendix]{Re:60}) concerns a soap film
which retracts to its boundary.

\begin{figure}
\includegraphics[width=0.95\textwidth]{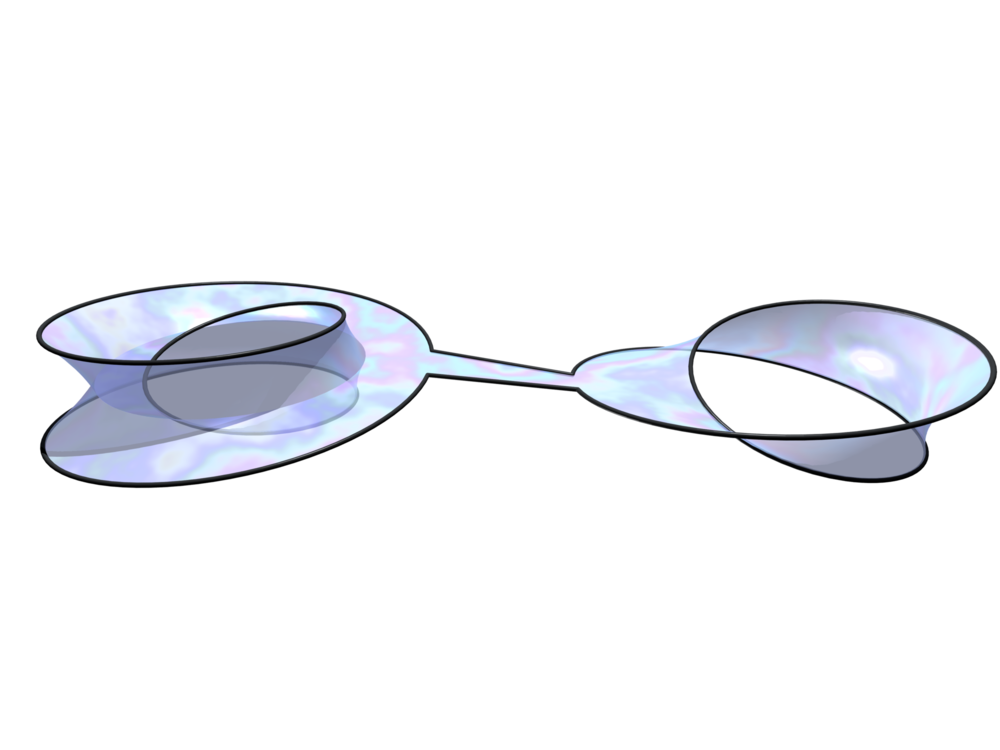}
\caption{\small{A minimal film that retracts to its boundary.
Image provided by E. Paolini.
An example of a film that \emph{deformation} retracts to its
boundary can be found in \cite[fig. 3]{Mo:93}, the same example
can also be found in \cite[fig. 14]{Br:95}}}
\label{fig:retract}
\end{figure}

\begin{Example}\label{exa:triple_moebius_band}\rm 
Let $S$ be the  curve of Figure \ref{fig:retract}: we would like
to consider 
the soap film of the figure as a cut, but in order to construct a 
consistent triple cover, this is not sufficient. Indeed, we add 
an invisible wire in the form of a loop $C$ circling around the 
Mo\"ebius strip on the right; next  we consider as a cut the
union of the soap film in the figure and a disk bounded by $C$.
Of course, this cut has a selfintersection along a diamater
of the disk. Now, take as usual three copies $1,2,3$ of the 
cutting surface and glue them using the permutations as follows:
\begin{itemize}
\item[] $(2~3)$ when crossing the disk bounded by $C$;
\item[] $(1~ 2~ 3)$ on the remaining part of the cut.
\end{itemize}

Observe that the part of the cut on the right hand side is not
orientable: the invisible wire acts in such a way to revert
the cyclic permutation $(1~ 2~ 3)$ when crossing the disk. 

It turns out that a presentation of 
the fundamental group of $M = \R^3 \setminus (S\cup C)$ is
$$
\pi_1(M) =  <a,b ; abab = baba>, 
$$
where $a$ corresponds to a small loop circling around 
$S$, and $b$ corresponds to a short loop circling
around the invisible wire $C$.

The abstract definition of the cover is obtained by considering
the homomorphism $\varphi : \pione \to \SSS_3$ that maps
$a$ to $(1~2~3)$  and $b$
to $(2~3)$\footnote{One verifies that  
$\varphi$ is well defined with respect to 
the relation of the presentation.}. A word belongs to $H < \pione$ 
whenever it consists of the words of $\pione$ that are mapped
through $\varphi$ in a permutation of $\{1,2,3\}$ which fixes $1$: namely,
either the identity $()$ or the transposition $(2~3)$.


\end{Example} 

\begin{Example}\label{exa:octahedron}\rm 
Let  $\edges$ be 
the one-skeleton of a regular octahedron. 
The fundamental group of $\base = \R^3 \setminus \edges$ 
is a free group of rank $5$.
After suitable orientation, each of the $12$ edges of the octahedron can be associated
to an element of $\pione$ corresponding to a loop from the base point (at infinity) that
circles once in the positive sense around it.

Imposing a strong wetting condition \cite{BePaPa:17} at all edges for a cover with three sheets amounts in
forcing the permutation of sheets corresponding to a positive loop around that edge to
be either $(1~2~3)$ or its inverse $(1~3~2)$.
Upon possibly reversing the orientation of some edge we can assume all such permutations to
be $(1~2~3)$.

Local triviality of the cover at points near a vertex then corresponds in requiring that
exactly two of the four edges concurring at that vertex to be ``incoming'', the other two
being ``outgoing''.

A choice of the orientation of the edges consistent with the requirement above
corresponds to travel clockwise along the boundary edges of four of
the eight faces selected in a checkerboard fashion.
The resulting soap film in Figure \ref{fig:octahedron} (top-left) simply consists in
those four faces or on the four remaining faces.

Another consistent choice of orientation consists in travelling around the three diametral
squares in a selected direction.
Two relative minimizers corresponding to this choice are shown in Figure \ref{fig:octahedron}
(top-right and bottom),
the latter consists in a tube-shaped surface with six lunettes attached along six triple curves.

It turns out that there are at least two other non isomorphic $3$-sheeted covers of the same base space,
which however seem not to provide minimizers different from the ones
described above.

\begin{figure}
\includegraphics[width=0.45\textwidth]{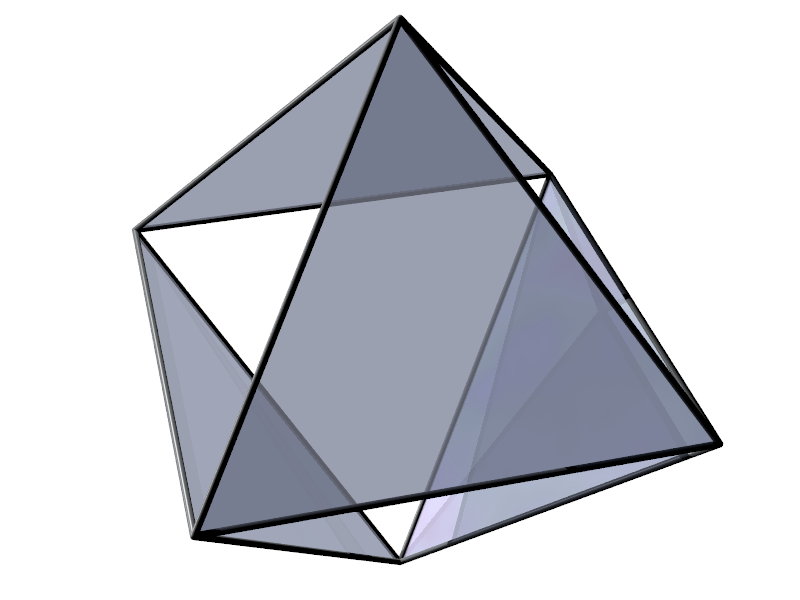}
\includegraphics[width=0.51\textwidth]{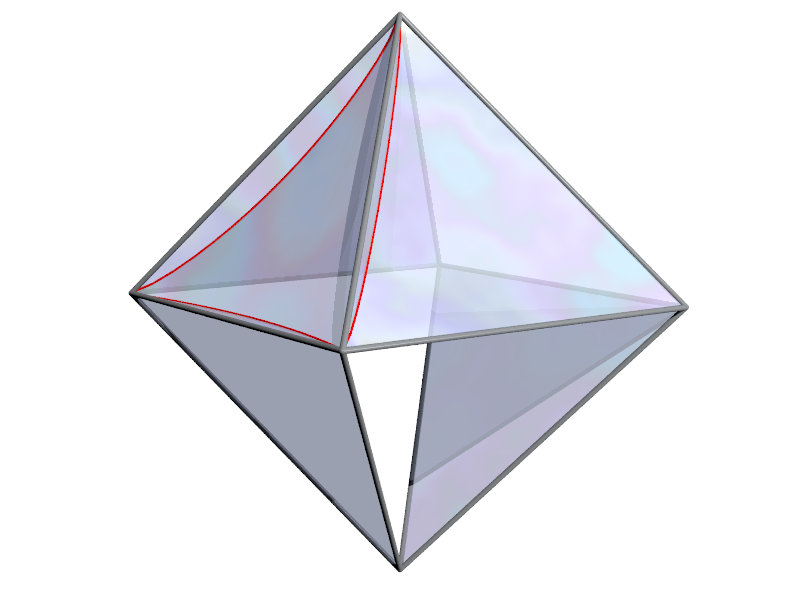}
\\
\includegraphics[width=0.48\textwidth]{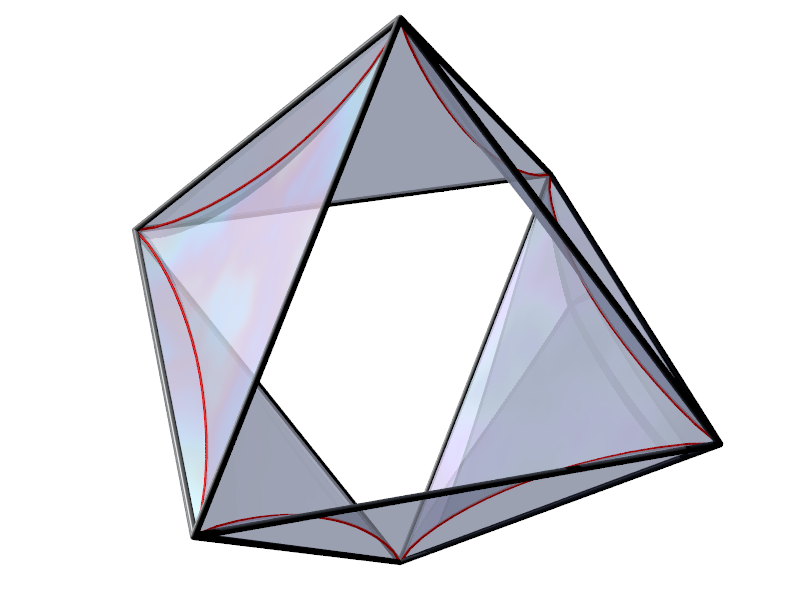}
\caption{\small{Three examples of non-simply connected minimal films spanning the boundary of
a regular octahedron.
Top-left: trivial nonconnected surface consisting in four of the eight faces;
Top-right: surface obtained by starting from five of the eight faces, the result
consists in an isolated triangular face $F$ (after removing its boundary)
plus a film with three triple curves wetting
all the edges of the octahedron that are not 
edges of $F$;
Bottom: surface obtained by starting from six of the eight faces.
Note the presence of six triple curves.
Images obtained using the \texttt{surf} code by E. Paolini.}}
\label{fig:octahedron}
\end{figure}

\end{Example}


\end{document}